\documentclass[12pt]{article}

\usepackage[colorlinks=true, pdfstartview=FitV, linkcolor=blue, citecolor=blue, urlcolor=blue]{hyperref}
\usepackage[justification=centering]{caption}
\usepackage{amssymb,amsmath, amscd,mathdots}
\usepackage{times, verbatim}
\usepackage{graphicx}

\DeclareFontFamily{OT1}{rsfs}{}
\DeclareFontShape{OT1}{rsfs}{n}{it}{<-> rsfs10}{}
\DeclareMathAlphabet{\mathscr}{OT1}{rsfs}{n}{it}

\newtheorem{theorem}{Theorem}

\newtheorem{lemma}[theorem]{Lemma}

\newtheorem{remark}[theorem]{Remark}
\newenvironment{proof}{\noindent {\bf Proof:}}{$\Box$ \vspace{2 ex}}

\newcounter{nootje}
\def\Z{{\mathbb Z}}

\def\GOE{{\rm GOE}}

\def\R{{\mathbb R}}
\def\F{{\mathbb F}}

\def\Q{{\mathbb Q}}

\def\C{{\mathcal C}}

\def\Z{{\mathbb Z}}

\def\F{{\mathbb F}}
\def\Q{{\mathbb Q}}
\def\C{{\mathbb C}}

\newcommand{\D}{\mathrm{d}}

\title{What is the probability that a random integral quadratic form
  in $n$~variables has an integral zero?}

\author{M.~Bhargava, J.~E.~Cremona, T.~A.~Fisher, N.~G.~Jones,
  and J.~P.~Keating}

\begin{document}

\maketitle

\begin{abstract}
We show that the density of quadratic forms in $n$ variables over
$\Z_p$ that are isotropic is a rational function of $p$, where the
rational function is independent of $p$, and we determine this
rational function explicitly.  When real quadratic forms in $n$
variables are distributed according to the Gaussian Orthogonal
Ensemble (\GOE) of random matrix theory, we determine explicitly the probability that a random
such real quadratic form is isotropic (i.e., indefinite).

As a consequence, for each $n$, we determine an exact expression for
the probability that a random {integral} quadratic form in $n$
variables is isotropic (i.e., has a nontrivial zero over $\Z$), when
these integral quadratic forms are chosen according to the GOE distribution.  In particular, we find an exact expression for
the probability
that a random integral quaternary quadratic form is isotropic;
numerically, this probability of isotropy is approximately $98.3\%$.
\end{abstract}


\section{Introduction}

An {integral quadratic form} $Q$ in $n$ variables is a homogeneous
quadratic polynomial
\begin{equation}\label{Qdef}
Q(x_1,x_2,\ldots,x_n) = \sum_{1\leq i\leq j\leq n} c_{ij}x_ix_j,
\end{equation}
where all coefficients $c_{ij}$ lie in $\Z$.  The quadratic form $Q$
is said to be {\it isotropic} if it represents 0, i.e., if there
exists a nonzero $n$-tuple $(k_1,\ldots,k_n)\in\Z^n$ such that
$Q(k_1,\ldots,k_n)=0$.  We wish to consider 
the question: what is the probability that a random integral
quadratic form in $n$ variables is isotropic?  

In this paper, we give a complete answer to this question for all $n$,
when integral quadratic forms in $n$ variables are chosen according to
the Gaussian Orthogonal Ensemble (\GOE) of random matrix theory
\cite{AGZ}.  In particular, in the most interesting case $n=4$, we
show that the probability that a random integral quaternary quadratic
form is isotropic is given by
\begin{equation}\label{quatprob}
\left(\frac12 + \frac{\sqrt{2}}{8}+\frac1\pi\right)  \prod_p\left(1 - \displaystyle\frac{p^3}
{4(p+1)^2(p^4+p^3+p^2+p+1)}\right)\approx 98.3\%.
\end{equation}

More precisely, let $D$ be a piecewise smooth rapidly decaying
function on the vector space $\R^{n(n+1)/2}$ of real quadratic forms
in $n$ variables (i.e., $D(x)$ and all its partial derivatives 
are $o(|x|^{-N})$ for all $N>0$), and assume that $\int_Q D(Q)dQ=1$; we call such a function
$D$ a {\it nice distribution} on the space of real $n$-ary quadratic
forms.  Then we define the probability, with respect to the
distribution $D$, that a random integral $n$-ary quadratic form $Q$
has a property~$P$ by
\begin{equation}\label{probdef2}
\lim_{X\to\infty} \frac{\sum_{Q \mbox{\scriptsize{ integral, with
        property $P$}}} D(Q/X)}
{\sum_{Q \mbox{ \scriptsize integral}} D(Q/X)},
\end{equation}
if the limit exists.  Let $\rho^D_n$ denote the probability with
respect to the distribution $D$ that a random integral quadratic form
in $n$ variables is isotropic.  If $D=\GOE$ is the distribution on the
space of $n\times n$ symmetric matrices given by
$\frac{1}{\sqrt{2}}(A+A^t)$, where each entry of the matrix $A$ is an
identical and independently distributed Gaussian---i.e., the Gaussian
Orthogonal Ensemble---then we use $\rho_n:=\rho^\GOE_n$ to denote the
probability, with respect to the \GOE\ distribution, that a random
$n$-ary quadratic form over~$\Z$ is isotropic.

We wish to explicitly determine the probability $\rho_n$ that a random
$n$-ary quadratic form over $\Z$, with respect to the
\GOE\ distribution, is isotropic, i.e., has a nontrivial zero over
$\Z$. To accomplish this, we first recall the Hasse--Minkowski
Theorem, which states that a quadratic form over $\Z$ is isotropic if
and only if it is isotropic over $\Z_p$ for all $p$ and over $\R$.
For any distribution $D$ as above, let $\rho_n^D(p)$ denote the
probability that a random integral quadratic form, with respect to the
distribution $D$, is isotropic over $\Z_p$, and let $\rho_n^D(\infty)$
denote the probability that it is isotropic over $\R$ (i.e., is
indefinite).  Then it is not hard to show (for the details, see
Section~\ref{localglobal}) that $\rho_n(p)=\rho_n^D(p)$ is independent
of $D$, and is simply given by the probability that a random $n$-ary
quadratic form over $\Z_p$, with respect to the usual additive measure
on $\Z_p^{n(n+1)/2}$, is isotropic over $\Z_p$.  Moreover, we will
also show in Section~\ref{localglobal} that the probability
$\rho_n^D(\infty)$ that a random \textit{integral} quadratic form is
isotropic over~$\R$ is equal to the probability that a random
\textit{real} quadratic form (with respect to the same distribution
$D$) is indefinite.

 For any distribution $D$ as above, it can be deduced from the work of
 Poonen and Voloch~\cite{PV}, together with the Hasse--Minkowski
 Theorem, that:
\begin{theorem}\label{productformula}
The probability $\rho_n^D$ that a random (with respect to the
distribution~$D$) integral quadratic form in $n$ variables is
isotropic is given by the product of the local probabilities:
\begin{equation}\label{pfeq}
\rho^D_n=\rho_n^D(\infty)\prod_p\rho_n(p).
\end{equation}
\end{theorem}
See~Section~\ref{localglobal} for details.
Hence, to determine $\rho_n^D$, it suffices to determine  $\rho_n^D(\infty)$ and $\rho_n(p)$ for all $p$.

We treat first the probability $\rho_n(p)$ that a random $n$-ary
quadratic form over $\Z_p$ is isotropic.  Our~main result here is
that, for each $n$, the quantity $\rho_n(p)$ is given by a fixed
rational function in $p$ that is independent of~$p$ (this even
includes the case $p=2$), and we determine these rational functions
explicitly.  Specifically, we prove the following theorem:

\begin{theorem}\label{mainlocal} 
  Let $\rho_n(p)$ denote the probability that a quadratic form in
  $n$ variables over $\Z_p$ is isotropic. Then 
\[ \rho_1(p)=0, \quad \rho_2(p) = \frac12, \quad \rho_3(p) = 1 - \frac{p}{2(p+1)^2}, \quad  \rho_4(p) = 1 - \frac{p^3}
{4(p+1)^2(p^4+p^3+p^2+p+1)},\]
and $\rho_n(p) = 1$ for all $n \ge 5$.
\end{theorem}
Our method of proof for Theorem~\ref{mainlocal} is uniform in $n$, and
relies on establishing certain recursive formulae for densities of
local solubility for certain subsets of $n$-ary quadratic forms
defined by their behavior modulo powers of~$p$. In particular, we
obtain a new recursive proof of the well-known fact that every $n$-ary
quadratic form over $\Q_p$ is isotropic when $n\geq 5$.  See
Section~\ref{sec:Zp} for details.


We turn next to the probability $\rho_n(\infty)=\rho_n^\GOE(\infty)$
that a real $n$-ary quadratic form is isotropic.  Closed form
expressions for $\rho_n(\infty)$ for $n\le3$
were first given by Beltran in \cite[(7)]{DM2}, while in \cite{DM} it was argued that
$1-\rho_n(\infty)$ decays like $e^{-n^2(\log 3)/4}$ as $n\rightarrow\infty$.

In Section~\ref{sec:R}, we show how to obtain an exact formula for
$\rho_n(\infty)$ for any given $n$.  More precisely, using the de
Bruijn identity~\cite{debruijn} for calculating certain determinantal
integrals, we express $\rho_n(\infty)$ as the Pfaffian of an explicit
$n'\times n'$ matrix, where $n':=2\lceil n/2 \rceil$, whose entries
are given in terms of values of the gamma and incomplete beta
functions at integers and half-integers.  Indeed, let $\Gamma$ denote
the usual gamma function $\Gamma(s)=\int_{0}^\infty x^{s-1}e^{-x} \D
x$, and $\beta_{t}$ the usual incomplete beta function $\beta_{t}(i,j)
= \int_0^{t} x^{i-1}(1-x)^{j-1} \D x.$ \label{sec:beta} Then we have
the following theorem giving expressions for $\rho_n(\infty)$:

\begin{theorem}\label{mainreal}
Let $n\geq 1$ be any integer, and define $n':=2 \lceil n/2 \rceil$.
When real $n$-ary quadratic forms are chosen according to the
$n$-dimensional Gaussian Orthogonal Ensemble, the probability of
isotropy over $\R$ is given by
\begin{equation}\rho_{n}(\infty)=1-\frac{{\rm
      Pf}(A)}{2^{(n-1)(n+4)/4}\prod_{m=1}^n\Gamma(\frac m
    2)},\label{theorem3}
\end{equation}
where $A$ is the $n'\times n'$ skew-symmetric matrix whose
$(i,j)$-entry $a_{ij}$ is given for $i<j$ by
\begin{align}\label{pfaffianformula}
a_{ij} = \begin{cases}
\,2^{i+j-2}\Gamma(\frac{i+j}{2})\left(\beta_{\frac{1}{2}}(\frac{i}{2},\frac{j}{2})- \beta_{\frac{1}{2}}(\frac{j}{2},\frac{i}{2})\right) & \mbox{if }~ i<j\leq n, \\[.075in]
\,2^{i-1}\Gamma(\frac{i}{2})  & \mbox{if }~ i<j=n+1.\\[.075in]
\end{cases} 
\end{align}
$($Note that the second case in $(\ref{pfaffianformula})$ arises only when $n$ is odd.$)$
\end{theorem}

Theorem~\ref{mainreal} allows one to calculate $\rho_n(\infty)$
exactly in closed form for any given~$n$.  In particular, it follows
from the Pfaffian
representation in Theorem~\ref{mainreal} that $\rho_n(\infty)$ is a
polynomial in $\pi^{-1}$ of degree at most $\lfloor\frac{n+1}4\rfloor$
with coefficients in $\Q(\sqrt2)$ (see~Remark~\ref{degrmk}).  In Table~\ref{table:GOE-results},
we give the resulting formulae for $\rho_n(\infty)$ for all $n\le8$, and
also provide numerical approximations.  (For any $n>8$, we have
$\rho_n(\infty)\approx 1$ to more than 10 decimal places!)

\begin{table}[ht]
\renewcommand{\arraystretch}{1.2}
$$
 \begin{array}{|r||c|c|}\hline n & \rho_{n}(\infty)= & \rho_{n}(\infty)\approx \\\hline\hline 1 & 0 & 0 \\\hline 2 & {\sqrt{2}}/2 & 0.7071067811 \\\hline 3 & {1}/{2} + {\sqrt{2}}\,{\pi^{-1}} & 0.9501581580 \\\hline 4 &  1/2 + \sqrt{2}/8 + {\pi^{-1}} & 0.9950865814 \\\hline 5 & 3/4 +(2/3+\sqrt2/12)\pi^{-1}  & 0.9997197706  \\\hline 6 & 3/4+7\sqrt2/64+(37/48-\sqrt2/3)\pi^{-1}&0.9999907596 \\\hline 7 & 7/8 + (47/120+109\sqrt2/480)\pi^{-1}-(32\sqrt2/45)\pi^{-2} & 0.9999998239 \\\hline 8 &
7/8+9\sqrt2/256+(2377/3840-53\sqrt2/480)\pi^{-1}-(32/45)\pi^{-2}  & 0.9999999980\\\hline 
\end{array}
$$
 \caption{Probability $\rho_{n}(\infty)$ that a random $n$-ary quadratic form over $\R$ from the GOE distribution is isotropic, for $n\leq 8$.}
   \label{table:GOE-results}
\end{table}

Combining Theorems~\ref{productformula}, \ref{mainlocal}, and \ref{mainreal}, we finally obtain the following theorem giving the
probability $\rho_n$ that a random integral quadratic form in $n$
variables has an integral zero:

\begin{theorem}\label{main}
Let $D$ be any nice (i.e., piecewise smooth and rapidly decaying)
distribution.  Then the probability $\rho_n^D$ that a random integral
quadratic form in $n$ variables with respect to the distribution $D$
is isotropic is given by
$$\rho_n^D = \left\{\begin{array}{cl}
0 & \mbox{if $n\leq 3$}; \\[.125in]
\rho_4^D(\infty)\displaystyle\prod_p\left(1 - \displaystyle\frac{p^3}
{4(p+1)^2(p^4+p^3+p^2+p+1)}\right) & \mbox{if $n=4$}; \\[.25in]
\rho_n^D(\infty) & \mbox{if $n\geq 5$}.
\end{array}\right.$$
If $D={\GOE}$ is the \GOE\ distribution, then the quantities
$\rho_n(\infty)=\rho_n^D(\infty)$ are as given in
Theorem~\ref{mainreal}.
\end{theorem}
In particular, when $D=\GOE$, we have $\rho_n=0$ for $n=1$, 2, and 3,
while for $n=4$ we obtain the expression (\ref{quatprob}) for
$\rho_4$.  For $n\geq 5$, we have $\rho_n=\rho_n(\infty)$, and so the
values of $\rho_n$ are as given by Theorem~$\ref{mainreal}$.
Theorem~\ref{main} shows that $n=4$ is in a sense the most interesting
case, as all places play a nontrivial role in the final answer.

It is also interesting to compare how the probabilities change if
instead of the \GOE\ we use the uniform distribution U on quadratic
forms, where each coefficient of the quadratic form is chosen
uniformly in the interval $[-1/2,1/2]$.  While the quantities
$\rho_n^{\rm U}(\infty)$ can easily be expressed as explicit definite
integrals, it seems unlikely that they can be evaluated in compact and
closed form for general $n$ in this case.  Using numerical
integration, or a Monte Carlo approximation, 
we can compute $\rho_n^{\rm U}(\infty)\approx 0,$ $0.627$, $0.901$,
$0.982$, $0.998$, and $>0.999$ for $n=1$, 2, 3, 4, 5, and 6,
respectively.  It is known (see, e.g., \cite[Theorem.~2.3.5]{AGZ}) that
$1-\rho_n^{\rm U}(\infty)$ decays faster than $e^{-cn}$ for some
constant $c>0$; the actual rate of decay is likely even faster.

In particular, we have $\rho_4^{\rm U}=\rho_4^{\rm
  U}(\infty)\prod_p\rho_4(p)\approx 97.0\%$, which is slightly smaller
than the \GOE\ probability $\rho_4^\GOE\approx 98.3\%$.
%
%
%
%
We summarize the values of $\rho^D_n$, and provide numerical values in
the cases of the uniform and \GOE\ distributions, in
Table~\ref{table:distributions}.

\begin{table}
\renewcommand{\arraystretch}{1.5}
$$\begin{array}{|r||c|c|c|}\hline
n & \rho^D_n &\rho_n^{\rm U} & \rho_n \\ \hline\hline
1 & 0 & 0 & 0 \\ \hline
2 & 0 & 0 & 0 \\ \hline
3 & 0 & 0 & 0 \\ \hline
4 & \,\rho_4^D(\infty)\prod_p\Bigl(1 - \frac{p^3}
{4(p+1)^2(p^4+p^3+p^2+p+1)}\Bigr)\, & \approx 97.0\% & \approx 98.3\%  \\ \hline
5 & \rho_5^D(\infty) & \approx 99.8\% & >99.9\%  \\ \hline\hline
\geq6 & \rho_n^D(\infty) & >99.9\% & >99.9\%  \\ \hline
\end{array}
$$
\vspace{-.235in}
\begin{center}
\caption{Probability that a random integral quadratic form in $n$
  variables is isotropic, for a general distribution $D$, for 
the uniform distribution, and for the \GOE\ distribution.}
\label{table:distributions}
\end{center}
\vspace{-.2in}
\end{table}

\goodbreak

This paper is organized as follows.  In Section~\ref{localglobal}, we
prove the product formula in Theorem~\ref{productformula}.  The
theorem is known in the case of the uniform distribution ${\rm U}$ (or
indeed any uniform distribution supported on a box) for any $n\geq 4$
by the work of Poonen and Voloch~\cite{PV}, which in turn depends on
the Ekedahl sieve~\cite{Ek}.  To complete the proof of
Theorem~\ref{productformula}, we first prove directly that both sides
of~(\ref{pfeq}) are equal to 0 for $n\leq 3$.  For $n\geq 4$, we 
prove that (\ref{pfeq}) is true for a general nice distribution $D$ by
approximating $D$ by a finite weighted average of uniform box
distributions, where the result is already known.  The condition that
$D$ is rapidly decreasing (as in the case of $D=\GOE$) plays a key
role in the proof; indeed, we show how counterexamples to (\ref{pfeq})
can be constructed when this condition does not hold.

In Section~\ref{sec:Zp}, we then prove Theorem~\ref{mainlocal}, i.e., we determine for each $n$ the exact $p$-adic density of $n$-ary quadratic forms over $\Z_p$ that are isotropic.  The outline of the proof is as follows.  First, we note that a quadratic form in $n$ variables defined over~$\Z_p$
can be anisotropic only if its reduction modulo~$p$ has either two
conjugate linear factors over $\F_{p^2}$ or a repeated
linear factor over~$\F_p$.  We first compute the probability of each of these
cases occurring, which is elementary.  We then determine the
probabilities of isotropy in each of these two cases by developing certain recursive
formulae for these probabilities, in terms of other suitable quantities,
which allow us to solve and obtain exact algebraic 
expressions for these probabilities for each value of $n$.
We note that our general argument shows 
in particular 
that quadratic forms 
in~$n\ge5$ variables over~$\Q_p$ are always isotropic, thus yielding a 
new recursive proof of this well-known fact.

Finally, we prove Theorem~\ref{mainreal} in Section~\ref{sec:R}, i.e.,
we determine for each $n$ the probability that a random real $n$-ary quadratic form
from the \GOE\ distribution is indefinite.  We accomplish this by
first expressing, as a certain determinantal integral, the probability
that an $n\times n$ symmetric matrix from the \GOE\ distribution has
all positive eigenvalues.  We then show how this determinantal
integral can be evaluated using the de Bruijn
identity~\cite{debruijn}, allowing us to obtain an expression for the
probability of positive definiteness in terms of the Pfaffian of an
explicit skew-symmetric matrix $A$, as given in
Theorem~\ref{mainreal}.  We note that the values of these
probabilities were known previously for $n\leq 3$
(cf.\ \cite[(7)]{DM2}).

We end this introduction by remarking that the analogues of Theorems~\ref{mainlocal} and \ref{main} also hold over a
general local or global field, respectively.  Here, we define global
densities of quadrics as in \cite[\S4]{PV}; more general densities
with respect to ``nice distributions'' could also be defined in an
analogous manner. 
Indeed, the analogue of
Theorem~\ref{productformula} holds (with the identical proof), where
the product on the right hand side of (\ref{pfeq}) should be taken
over all finite and infinite places of the number field (the densities at
the complex places are all equal to 1, since all quadratic forms over $\C$ are isotropic).  
Theorem~\ref{mainlocal} also holds over any finite extension of $\Q_p$, 
with the same proof, provided that when making substitutions in the
proofs we replace $p$ by a uniformiser, and when computing
probabilities we replace $p$ by the order of the residue field.

\pagebreak

\section{The local product formula: Proof of Theorem~\ref{productformula}}\label{localglobal}

Let $D$ be any nice (piecewise smooth and rapidly decaying)
distribution.  Our aim in this section is to prove the following three
assertions from the introduction:

\begin{itemize}
\item[ (a)] $\rho_n^D(p)$ is equal to the probability $\rho_n(p)$ that
  a random $n$-ary quadratic form over $\Z_p$, with respect to the
  usual additive measure on $\Z_p^{n(n+1)/2}$, is isotropic over
  $\Z_p$;
\item[(b)] $\rho_n^D(\infty)$ is equal to the probability that a
  random $n$-ary quadratic form over $\R$, with respect to the
  distribution $D$, is indefinite; and
\item[(c)] $\rho_n^D=\rho_n^D(\infty)\prod_p \rho_n(p)$ (i.e.,
  Theorem~\ref{productformula} holds).
\end{itemize}

Items (a) and (b) are trivial in the case that $D={\rm U}$ is the uniform distribution, or more generally when $D$ is any distribution ${\rm U}(\vec a,\vec b)$ that is constant on a box $[\vec a,\vec b]:=[a_1,b_1]\times\cdots\times[a_{n(n+1)/2},b_{n(n+1)/2}]$ and 0 outside this box; here $\vec a=(a_1,\ldots,a_{n(n+1)/2})$ and $\vec b=(b_1,\ldots,b_{n(n+1)/2})$ are vectors in $\R^{n(n+1)/2}$ such that $a_i<b_i$ for all $i$.  

Meanwhile, Theorem~\ref{productformula} for $n\geq4$, in the case that $D$ is the uniform distribution ${\rm U}$, follows from the work of Poonen and Voloch~\cite[Theorem~3.6]{PV} (which establishes the product formula for the probability that an integral quadratic form  with respect to the distribution $D$ is locally soluble), together with the Hasse--Minkowski Theorem (which states that a quadratic form is isotropic if and only if it is locally soluble). In fact, the proof of \cite[Theorem~3.6]{PV} (which in turn relies on the Ekedahl's sieve~\cite{Ek}) immediately adapts to the case where $D={\rm U}(\vec a,\vec b)$ without essential change.  

To show that Theorem~\ref{productformula} holds also when $D={\rm U}(\vec a,\vec b)$ and $n\leq 3$, it suffices to prove that in this case both sides of (\ref{pfeq}) are equal to 0.  To see this, we may use Theorem~\ref{mainlocal}, which does not rely on the results of this section, and which states that the probability that a random $n$-ary quadratic form over $\Z_p$ is isotropic is equal to $\rho_n(p)=0$, $1/2$, or $1-p/(2(p+1)^2)$ for $n=1$, 2, or 3, respectively.  This immediately implies that the right hand side of (\ref{pfeq}) is zero.  To see that the left hand side of (\ref{pfeq}) is zero, we note that if a quadratic form over $\Z$ is isotropic, then it must be isotropic over $\Z_p$ for all $p$ (the easy direction of the Hasse--Minkowski Theorem).  By the Chinese Remainder Theorem, the (limsup of the) probability $\rho_n^D$ that a random integral $n$-ary  
quadratic form is isotropic with respect to the distribution $D={\rm U}(\vec a,\vec b)$ is at most
$$ \prod_{p<Y} \rho_n(p)$$
for any $Y>0$. Letting $Y$ now tend to infinity shows that $\rho_n^D=0$ for $n=1$, 2, or 3, i.e., the left hand side of (\ref{pfeq}) is also zero. 

Thus we have established items (a)--(c), for all $n$, in the case that $D={\rm U}(\vec a,\vec b)$ is a constant distribution supported on a box $[\vec a,\vec b]$.  Clearly (a)--(c) then must hold also for any finite weighted average of such box distributions ${\rm U}(\cdot,\cdot)$.

To show that (a)--(c) hold for general nice distributions $D$, we make use of the following elementary lemma regarding integration of rapidly decaying functions.

\begin{lemma}\label{rdlemma}
Let $f$ be any piecewise smooth rapidly decaying function on $\R^m$.  Then 
\begin{equation}\label{rs}
\int f(y)\D y =\lim_{X\to\infty} \frac1{X^m}\sum_{y\in\Z^m}f(y/X).
\end{equation}
\end{lemma}

\begin{proof}
For any $N>0$, let $f_N(y)$ be equal to $f(y)$ if $|y|\leq N$, and 0 otherwise.  Then $f_N$ is piecewise smooth with bounded support, and so is Riemann integrable.  Thus we have
\begin{equation}\label{rs2}
\int f_N(y)\D y =\lim_{X\to\infty} \frac1{X^m}\sum_{y\in\Z^m}f_N(y/X).
\end{equation}
Since $f$ is rapidly decreasing, for any $\varepsilon>0$ we may choose $N$ large enough so that 
$\int_{|y|> N}|f(y)|\D y<\varepsilon$ {\em and} $
(1/X^m)\sum_{y\in\Z^m,\,|y/X|>N} |f(y/X)|<\varepsilon$ for any $X\geq
1$.  For this value of $N$, the left hand side of (\ref{rs2}) is
within $\varepsilon$ of the left hand side of (\ref{rs}), while for
each $X\geq 1$, the expression in the limit on the right hand side of
(\ref{rs2}) is within $\varepsilon$ of the expression in the limit on
the right hand side of (\ref{rs}).  Since we have equality in
(\ref{rs2}), we conclude that the left hand side of (\ref{rs}) is
within $2\varepsilon$ of both the $\liminf_{X\to\infty}$ and the
$\limsup_{X\to\infty}$ of the expression in the limit of the right
hand side of (\ref{rs}).  Since $\varepsilon$ is arbitrarily small, we
have proven (\ref{rs}).
\end{proof}

Note that Lemma~\ref{rdlemma} does not necessarily hold if we drop the
condition that $f$ is rapidly decaying.  For example, if $f$ the
characteristic of a finite-volume region having a cusp going off to
infinity containing a rational line through the origin (and thus
infinitely many lattice points on that line), then the left hand side
of (\ref{rs}) is finite while the expression in the limit on the right
hand side of (\ref{rs}) is infinite for any rational value of $X$.

Lemma~\ref{rdlemma} implies in particular that 
\begin{equation}\label{D0}
\lim_{X\to\infty}\frac1{X^{n(n+1)/2}}{\sum_{Q {\rm {\;integral}}}D(Q/X)}=1
\end{equation}
for any nice distribution $D$.

Now any piecewise smooth rapidly decaying function can be approximated arbitrarily well by a finite linear combination of characteristic functions of boxes.  Let $D$ be a nice distribution. For any $\varepsilon>0$, 
we may find a nice distribution $D_\varepsilon$ that is a finite weighted average of box distributions ${\rm U}(\cdot,\cdot)$, such that
\begin{equation}\label{lteps}
\int |D(y)-D_\varepsilon(y)|\D y<\varepsilon.
\end{equation}
By Lemma~\ref{rdlemma}, we then have
\begin{equation}\label{ltepscor}
\lim_{X\to\infty} \frac1{X^{n(n+1)/2}}\sum_{Q{\rm \;integral}}|D(Q/X)-D_\varepsilon(Q/X)| < \varepsilon.
\end{equation}

To show that $\rho_n^D(p)=\rho_n(p)$, we note that
\begin{align}\label{D1}
\rho_n(p)=\rho_n^{D_\varepsilon}(p)&=\lim_{X\to\infty}\frac{\sum_{Q {\rm {\:integral, isotropic/}}\Z_p}D_\varepsilon(Q/X)}{\sum_{Q {\rm {\;integral}}}D_\varepsilon(Q/X)}\\
&=\lim_{X\to\infty}\frac{\sum_{Q {\rm {\:integral, isotropic/}}\Z_p}D_\varepsilon(Q/X)}{X^{n(n+1)/2}}\label{D12}\\
&=\lim_{X\to\infty}\frac{\sum_{Q {\rm {\:integral, isotropic/}}\Z_p}D(Q/X)+E(X,\varepsilon)}{X^{n(n+1)/2}}\label{D14}\\
&=\lim_{X\to\infty}\frac{\sum_{Q {\rm {\:integral, isotropic/}}\Z_p}D(Q/X)+E(X,\varepsilon)}{ \sum_{Q {\rm {\;integral}}}D(Q/X)},\label{D15}
\end{align}
where for sufficiently large $X$ we have $|E(X,\varepsilon)|<\varepsilon X^{n(n+1)/2}$ by (\ref{ltepscor}); here the first equality follows because $D_\varepsilon$ is a finite weighted average of box distributions ${\rm U}(\cdot,\cdot)$, 
the second equality follows from the definition~(\ref{probdef2}), and the third and fifth equalities follow from (\ref{D0}).
Letting $\varepsilon$ tend to~0 in (\ref{D15}) now yields $\rho_n(p)=\rho_n^D(p)$, proving item (a) for general nice distributions $D$.

Analogously, we have 
\begin{align}\label{D2}
\int_{Q {\rm {\:isotropic/}}\R}D_\varepsilon(Q)\D Q=
\rho_n^{D_\varepsilon}(\infty)&=\lim_{X\to\infty}\frac{\sum_{Q {\rm {\:integral, isotropic/}}\R}D_\varepsilon(Q/X)}{\sum_{Q {\rm {\;integral}}}D_\varepsilon(Q/X)}\\
&=\lim_{X\to\infty}\frac{\sum_{Q {\rm {\:integral, isotropic/}}\R}D(Q/X)+E'(X,\varepsilon)}{\sum_{Q {\rm {\;integral}}}D(Q/X)},\label{D22}
\end{align}
where again for sufficiently large $X$ we have $|E'(X,\varepsilon)|<\varepsilon X^{n(n+1)/2}$. By (\ref{lteps}), the leftmost expression in~(\ref{D2}) approaches $\int_{Q {\rm {\:isotropic/}}\R}D(Q)\D Q$ as $\varepsilon\to0$, while expression (\ref{D22}) approaches
$\rho_n^D(\infty)$ by definition~(\ref{probdef2}).  This thus proves item (b) for general nice distributions.  In particular, we have also proven that 
\begin{equation}\label{epslim}
\lim_{\varepsilon\to0} \rho_n^{D_\varepsilon}(\infty) =\rho_n^D(\infty).
\end{equation}

Finally, we have in a similar manner:
\begin{align}\label{D3}
\rho_n^{D_\varepsilon}(\infty)\prod_p 
\rho_n(p)=\rho^{D_\varepsilon}_n&=\lim_{X\to\infty}\frac{\sum_{Q {\rm {\:integral, isotropic/}}\Z}D_\varepsilon(Q/X)}{\sum_{Q {\rm {\;integral}}}D_\varepsilon(Q/X)}\\
&=\lim_{X\to\infty}\frac{\sum_{Q {\rm {\:integral, isotropic/}}\Z}D(Q/X)+E''(X,\varepsilon)}{\sum_{Q {\rm {\;integral}}}D(Q/X)},\label{D32}
\end{align}
where again for sufficiently large $X$ we have $|E''(X,\varepsilon)|<\varepsilon X^{n(n+1)/2}$. By (\ref{epslim}), the leftmost expression in~(\ref{D3}) approaches $\rho_n^{D}(\infty)\prod_p 
\rho_n(p)$ as $\varepsilon\to0$, while expression (\ref{D32}) approaches $\rho_n^D$ by definition.  We have proven also item (c) for general nice distributions, as desired.

\section{The density of $n$-ary quadratic forms over $\Z_p$ that are isotropic: Proof of Theorem~\ref{mainlocal}}\label{sec:Zp}

\subsection{Preliminaries on $n$-ary quadratic forms over $\Z_p$}

Fix a prime $p$. For any free $\Z_p$-module $V$ of finite rank, there
is a unique additive $p$-adic Haar measure $\mu_V$ on $V$ which we
always normalize so that $\mu_V(V)=1$.  All densities/probabilities
are computed with respect to this measure.  In this section, we take
$V=V_n$ to be the $n(n+1)/2$-dimensional $\Z_p$-module of $n$-ary
quadratic forms over $\Z_p$. We then work out the density $\rho_n(p)$
(i.e. measure with respect to $\mu_V$) of the set of $n$-ary quadratic
forms over $\Z_p$ that are isotropic.

We start by observing that a primitive $n$-ary quadratic form over
$\Z_p$ can be anisotropic only if, either: (I) the reduction
modulo~$p$ factors into two conjugate linear factors defined over a
quadratic extension of $\F_p$, or (II) the reduction modulo~ $p$ is a
constant times the square of a linear form over $\F_p$. Let
$\xi_1^{(n)}$ and $\xi_2^{(n)}$ be the probabilities of Cases I and
II, i.e. the densities of these two types of quadratic forms in
$V_n$. Then
\[ \xi_0^{(n)} = 1 - \xi_1^{(n)}- \xi_2^{(n)}-\frac1{p^{n(n+1)/2}} \]
is the probability that a form is primitive, but not in Cases I or II.
Let $\alpha_1^{(n)}$ (resp.\ $\alpha_2^{(n)}$) be the probability 
of isotropy for quadratic forms in Case I (resp.\ Case II).
Then
\[ \rho_n(p)=\xi_0^{(n)}+\xi_1^{(n)}\alpha_1^{(n)}+\xi_2^{(n)}\alpha_2^{(n)}+\frac1{p^{n(n+1)/2}}\rho_n(p),\]
implying that
\begin{equation}\label{rhoformula}
\rho_n(p)=\frac{p^{n(n+1)/2}}{p^{n(n+1)/2}-1}\bigl(\xi_0^{(n)}+\xi_1^{(n)}\alpha_1^{(n)}+\xi_2^{(n)}\alpha_2^{(n)}\bigr).
\end{equation}

\subsection{Some counting over finite fields}

Let $\eta_1^{(n)}$ (resp.\ $\eta_2^{(n)}$) be the probability that a
quadratic form is in
Case I (resp.\ Case II) given the ``point condition'' that the
coefficient of $x_1^2$ is a unit. Similarly, let $\nu_1^{(n)}$ be the
probability that a
quadratic form is in Case I given the ``line condition'' that the
binary quadratic form $Q(x_1,x_2,0,\ldots,0)$ is irreducible modulo
$p$.  Note that it is impossible to be in Case II given the line
condition, but we may also define $\nu_2^{(n)}=0$.  Set $\eta_0^{(n)}
= 1- \eta_1^{(n)} - \eta_2^{(n)}$ and $\nu_0^{(n)} = 1-
\nu_1^{(n)}-\nu_2^{(n)}= 1-\nu_1^{(n)}$.  The values of $\xi_j^{(n)}$,
$\eta_j^{(n)}$, $\nu_j^{(n)}$, are given by the following easy lemma.

\begin{lemma}\label{badformsdensity}
The probabilities that a random quadratic form over $\Z_p$
is in Case I or Case II are as follows.
\begin{itemize}
\item Case I (all; relative to point condition; relative to line condition)
\[ \xi_1^{(n)}=\frac{(p^n-1)(p^n-p)}{2(p+1)p^{n(n+1)/2}}; \qquad
\eta_1^{(n)}=\frac{p^{n-1}-1}{2p^{n(n-1)/2}}; \qquad
\nu_1^{(n)}=\frac{1}{p^{(n-1)(n-2)/2}}. \]

\item Case II (all; relative to point condition; relative to line condition)
\[\xi_2^{(n)}=\frac{p^n-1}{p^{n(n+1)/2}}; \qquad
\eta_2^{(n)}=\frac{1}{p^{n(n-1)/2}}; \qquad
\nu_2^{(n)}=0.\]
\end{itemize}
\end{lemma}

\begin{proof}
Case I: There are $(p^{2n}-1)/(p^2-1)$ linear forms over~$\F_{p^2}$ up
to scaling; subtracting the $(p^{n}-1)/(p-1)$ which are defined
over~$\F_p$, dividing by~$2$ to account for conjugate pairs and then
multiplying by $p-1$ for scaling gives $\frac{(p^n-1)(p^n-p)}{2(p+1)}$
Case I forms, and hence the value of~$\xi_1^{(n)}$.

Similarly, the number of Case I quadratic forms satisfying the point condition
is $(p^{2(n-1)}-p^{n-1})(p-1)/2$.  Dividing by the probability $1-1/p$
of the point condition holding gives $p^n(p^{n-1}-1)/2$ and hence the
value of~$\eta_1^{(n)}$.

Lastly, the number of Case I quadratic forms satisfying the line
condition is~$p^{2n-3}(p-1)^2/2$; dividing by the
probability~$\xi_1^{(2)}$ of the line condition holding gives
$p^{2n-1}$, and hence the value of~$\nu_1^{(n)}$.

Case II is similar and easier: the number of Case II quadratic forms
is $p^n-1$, of which $p^n-p^{n-1}$ satisfy the point condition and
none satisfy the line condition; the given formulae follow.
\end{proof}

\subsection{Recursive formulae}

We now outline our strategy for computing the densities~$\rho_n(p)$
using~(\ref{rhoformula}), by evaluating~$\alpha_j^{(n)}$ for $j=1,2$.
If a quadratic form is in Case I, then we may make a (density-preserving)
change of variables, transforming it so that its reduction is an
irreducible binary form in only two variables.  Now isotropy forces
the values of those variables, in any primitive vector giving a zero,
to be multiples of~$p$; so we may scale those variables by~$p$ and
divide the form by~$p$.  Similarly, if a form is in Case II, then we
transform it so that its reduction is the square of a single variable,
scale that variable and divide out.  After carrying out this process
once, we again divide into cases and repeat the procedure, which leads
us back to an earlier situation but with either the line or point
conditions, which we need to allow for.

To make this precise, we introduce some extra notation for the
probability of isotropy for quadratic forms which are in Case I or
Case II after the
initial transformation: let $\beta_1^{(n)}$ (resp.~$\beta_2^{(n)}$) be
the probability of isotropy given we are
in Case I (resp.\ Case II) after one step when the
original quadratic form was in Case I, and similarly $\gamma_1^{(n)}$
(resp.~$\gamma_2^{(n)}$) the probability of 
isotropy given we are in Case I (resp.\ Case II)
after one step when the original quadratic form was in Case II.

\begin{lemma}. \label{alphaformulas}

\begin{enumerate}
\item[\rm 1.] $\alpha_1^{(2)} = 0$, and for $n\ge3$,
\begin{equation*}
 \alpha_1^{(n)} = \xi_0^{(n-2)} + \xi_1^{(n-2)} \beta_1^{(n)}
+ \xi_2^{(n-2)} \beta_2^{(n)} + \frac{1}{p^{(n-1)(n-2)/2}} (\nu_0^{(n)}
 + \nu_1^{(n)} \alpha_1^{(n)} + \nu_2^{(n)} \alpha_2^{(n)}).
\end{equation*}
\item[\rm 2.] $\alpha_2^{(1)} = 0$, and for $n\ge2$,
\begin{equation*}
 \alpha_2^{(n)} = \xi_0^{(n-1)} + \xi_1^{(n-1)} \gamma_1^{(n)}
+ \xi_2^{(n-1)} \gamma_2^{(n)} + \frac{1}{p^{n(n-1)/2}} (\eta_0^{(n)}
 + \eta_1^{(n)} \alpha_1^{(n)} + \eta_2^{(n)} \alpha_2^{(n)} ).
\end{equation*}
\end{enumerate}
\end{lemma}


\begin{proof}
We have $\alpha_1^{(2)} = 0$ since a binary quadratic form that is
irreducible over~$\F_p$ is anisotropic. Now assume that $n \ge 3$, and
(for Case I) $Q(x_1, \ldots,x_n) \pmod p$ has two conjugate linear
factors.  Without loss of generality, the reduction modulo~$p$ is a
binary quadratic form in $x_1$ and $x_2$.  Now any primitive vector
giving a zero of~$Q$ must have its first two coordinates divisible
by~$p$, so replace $Q(x_1, \ldots,x_n)$ by $\frac{1}{p}Q(p x_1,p
x_2,x_3, \ldots, x_n)$. The reduction modulo~$p$ is now a quadratic
form in $x_3, \ldots,x_n$.  If the new~$Q$ is identically zero
modulo~$p$, then, after dividing it by~$p$, we obtain a new integral
form that lands in Cases I and II with probabilities $\nu_1^{(n)}$ and
$\nu_2^{(n)}$, respectively, since it satisfies the line condition;
otherwise, we divide into cases as before, with the probabilities of
being in each case given by $\xi_j^{(n-2)}$.

 The result for $\alpha_2^{(n)}$ is proved similarly: without loss of
 generality the reduction modulo~$p$ is a quadratic form in $x_1$
 only, we replace $Q(x_1, \ldots,x_n)$ by $\frac{1}{p}Q(p x_1,x_2,
 \ldots, x_n)$, whose reduction modulo~$p$ is a quadratic form in
 $x_2, \ldots,x_n$.  If the new~$Q$ is identically zero modulo~$p$,
 then, after dividing by~$p$, we have an integral form that lands in
 Cases I and II with probabilities $\eta_1^{(n)}$ and $\eta_2^{(n)}$,
 respectively, since it satisfies the point condition; otherwise, we
 divide into cases, with probabilities~$\xi_j^{(n-1)}$.
\end{proof}

It remains to compute $\beta_1^{(n)}$ (for $n \ge 4$), $\beta_2^{(n)}$
(for $n \ge 3$), $\gamma_1^{(n)}$ (for $n \ge 3$) and $\gamma_2^{(n)}$
(for $n \ge 2$).  Since $\xi_1^{(1)} = 0$, we do not need to compute
$\beta_1^{(3)}$ or~$\gamma_1^{(2)}$, which are in any case undefined.

\begin{lemma}.

\label{lem:betagamma}
\begin{enumerate}
\item[\rm (i)] If $n \ge 4$ then $\beta_1^{(n)} = \nu_0^{(n-2)}
 + \nu_1^{(n-2)} \beta_1^{(n)}$; also, $\beta_1^{(4)} = 0$.
\item[\rm (ii)] If $n \ge 3$ then $\beta_2^{(n)} = \nu_0^{(n-1)} + \nu_1^{(n-1)}
\gamma_1^{(n)}$; also, $\beta_2^{(3)} = 0$.
\item[\rm (iii)] If $n \ge 3$ then $\gamma_1^{(n)} = \eta_0^{(n-2)}
 + \eta_1^{(n-2)} \beta_1^{(n)} + \eta_2^{(n-2)} \beta_2^{(n)}$;
also, $\gamma_1^{(3)} = 0$.
\item[\rm (iv)] If $n \ge 2$ then $\gamma_2^{(n)} = \eta_0^{(n-1)} + \eta_1^{(n-1)}
\gamma_1^{(n)} + \eta_2^{(n-1)} \gamma_2^{(n)}$; also, $\gamma_2^{(2)} = 0$.
\end{enumerate} 
\end{lemma}

\begin{proof}
In Case I, the initial transformation leads to a quadratic form for
which the valuations of the coefficients satisfy\footnote{In this and
  the similar arrays which follow, we put into position $(i,j)$ the
  known condition on $v(a_{i,j})$, so the top left entry refers to the
  coefficient of $x_1^2$, the top right to $x_1x_n$ and the bottom
  right to $x_n^2$.}
\begin{equation}\label{array:caseI}
\begin{array}{ccccccc}
\ge 1 & \ge 1 & \ge 1 & \ge 1 & \ge 1 & \ldots & \ge 1 \\ & \ge 1 &
\ge 1 & \ge 1 & \ge 1 & \ldots & \ge 1 \\ & & \ge 0 & \ge 0 & \ge 0 &
\ldots & \ge 0 \\ & & & \ge 0 & \ge 0 & \ldots & \ge 0 \\ & & & & \ge
0 & \ldots & \ge 0 \\ & & & & & \ddots& \vdots \\ & & & & & & \ge 0
\end{array}
\end{equation}
and $\beta_1^{(n)}$ (resp.~$\beta_2^{(n)}$) are the probabilities of
isotropy given that the reduction modulo $p$ of the form in $x_3,x_4,\ldots,x_n$ is in Case~I
(resp.\ Case II).

Similarly, in Case II the initial transformation leads to
\begin{equation}\label{array:caseII}
\begin{array}{ccccccc}
=1 & \ge 1 & \ge 1 & \ge 1 & \ge 1 & \ldots & \ge 1 \\
& \ge 0 & \ge 0 & \ge 0 & \ge 0 & \ldots & \ge 0 \\
& & \ge 0 & \ge 0 & \ge 0 & \ldots & \ge 0 \\
& & & \ge 0 & \ge 0 & \ldots & \ge 0 \\
& & & & \ge 0 & \ldots & \ge 0 \\
& & & & & \ddots& \vdots \\
& & & & & & \ge 0
\end{array}
\end{equation}
and $\gamma_1^{(n)}$ (resp.~$\gamma_2^{(n)}$) are the probabilities of
isotropy given that the reduction modulo $p$ of the form in $x_2,x_3,\ldots,x_n$ is in Case~I
(resp.\ Case~II).

\vspace{.1in}\noindent
(i) To evaluate~$\beta_1^{(n)}$ we may assume, after a second linear
change of variables, that we have
\[
\begin{array}{ccccccc}
\ge 1 & \ge 1 & \ge 1 & \ge 1 & \ge 1 & \ldots & \ge 1 \\
      & \ge 1 & \ge 1 & \ge 1 & \ge 1 & \ldots & \ge 1 \\
      &       & \ge 0 & \ge 0 & \ge 1 & \ldots & \ge 1 \\
      &       &       & \ge 0 & \ge 1 & \ldots & \ge 1 \\
      &       &       &       & \ge 1 & \ldots & \ge 1 \\
      &       &       &       &       & \ddots& \vdots \\
      &       &       &       &       &        & \ge 1
\end{array}
\]
and that the reductions modulo~$p$ of both $\frac{1}{p} Q(x_1,x_2,0,
\ldots,0)$ and $Q(0,0,x_3,x_4,0, \ldots,0)$ are irreducible binary
quadratic forms.  Any zero of~$Q$ must satisfy $x_3\equiv
x_4\equiv0\pmod{p}$.  This gives a contradiction when~$n=4$, so that
$Q(x_1, \ldots, x_4)$ is anisotropic, and $\beta_1^{(4)}=0$.
Otherwise, replacing $Q(x_1,\ldots,x_n)$ by $\frac{1}{p} Q(x_3,x_4,p
x_1, p x_2,x_5, \ldots,x_n)$ brings us back to the situation
in~(\ref{array:caseI}).  Now, however, the line condition holds, so
that Cases I and II occur with probabilities $\nu_1^{(n-2)}$ and
$\nu_2^{(n-2)}=0$ instead of~$\xi_1^{(n-2)}$ and $\xi_2^{(n-2)}$.

\vspace{.1in}\noindent
(ii) To evaluate~$\beta_2^{(n)}$, we may assume that the valuations of
the coefficients satisfy
\[ \begin{array}{ccccccc}
\ge 1 & \ge 1 & \ge 1 & \ge 1 & \ldots & \ge 1 \\
      & \ge 1 & \ge 1 & \ge 1 & \ldots & \ge 1 \\
      &       &   = 0 & \ge 1 & \ldots & \ge 1 \\
      &       &       & \ge 1 & \ldots & \ge 1 \\
      &       &       &       & \ddots & \vdots \\
      &       &       &       &        & \ge 1
\end{array} \]
and that the reduction modulo~$p$ of $\frac{1}{p} Q(x_1,x_2,0,
\ldots,0)$ is an irreducible binary quadratic form.  If $n=3$ then $Q$
is anisotropic, and $\beta_2^{(3)}=0$.  Otherwise, replacing
$Q(x_1,\ldots,x_n)$ by $\frac{1}{p} Q(x_2,x_3,p x_1,x_4, \ldots,x_n)$
brings us back to the situation in~(\ref{array:caseII}) but with the
line condition, so that Cases I and II occur with probabilities
$\nu_1^{(n-1)}$ and $\nu_2^{(n-1)}$ instead of~$\xi_1^{(n-1)}$,
$\xi_2^{(n-1)}$.

\vspace{.1in}\noindent
(iii) For $\gamma_1^{(n)}$, we may assume that the valuations of the
coefficients satisfy
\[ \begin{array}{ccccccc}
  = 1 & \ge 1 & \ge 1 & \ge 1 & \ldots & \ge 1 \\
      & \ge 0 & \ge 0 & \ge 1 & \ldots & \ge 1 \\
      &       & \ge 0 & \ge 1 & \ldots & \ge 1 \\
      &       &       & \ge 1 & \ldots & \ge 1 \\
      &       &       &       & \ddots & \vdots \\
      &       &       &       &        & \ge 1
\end{array} \]
and the reduction of $Q(0,x_2,x_3,0, \ldots,0)$ modulo~$p$ is
irreducible. Any zero of~$Q$ now satisfies $x_2\equiv
x_3\equiv0\pmod{p}$.  When $n=3$ this gives a contradiction, so
$Q(x_1, x_2, x_3)$ is anisotropic, and $\gamma_1^{(3)}=0$.  Otherwise,
replacing $Q(x_1,\ldots,x_n)$ by $\frac{1}{p} Q(x_3, p x_1, p x_2,x_4,
\ldots,x_n)$ brings us back to the situation in~(\ref{array:caseI})
but with the point condition, so that Cases I and II occur with
probabilities $\eta_1^{(n-2)}$ and $\eta_2^{(n-2)}$.

\vspace{.1in}\noindent
(iv) Lastly, for $\gamma_1^{(n)}$, we may assume that the valuations
of the coefficients satisfy
\[ \begin{array}{ccccccc}
  = 1 & \ge 1 & \ge 1 & \ldots & \ge 1 \\
      &   = 0 & \ge 1 & \ldots & \ge 1 \\
      &       & \ge 1 & \ldots & \ge 1 \\
      &       &       & \ddots & \vdots \\
      &       &       &        & \ge 1 \rlap{.}
\end{array} \]
If $n=2$ then $Q(x_1, x_2)$ is anisotropic , and $\gamma_2^{(2)}=0$.
Otherwise, replacing $Q(x_1,\ldots,x_n)$ by $\frac{1}{p} Q(x_2, p
x_1,x_3, \ldots,x_n)$ brings us back to the situation
in~(\ref{array:caseII}) but with the point condition.
\end{proof}

\subsection{Conclusion}

Using Lemmas~\ref{badformsdensity} and~\ref{lem:betagamma} we can
compute~$\beta_j^{(n)}$ and~$\gamma_j^{(n)}$ for $j=1,2$ and all~$n$:
we first determine $\beta_1$ from Lemma~\ref{lem:betagamma}~(i), then
$\beta_2^{(n)}$ and $\gamma_1^{(n)}$ together using
Lemma~\ref{lem:betagamma}~(ii,iii), and finally $\gamma_2^{(n)}$
using Lemma~\ref{lem:betagamma}~(iv).  The following table gives the
result:
\[ \begin{array}{c|c|c|c|c}
 & \beta_1^{(n)} & \beta_2^{(n)} & \gamma_1^{(n)} & \gamma_2^{(n)} \\[.01in] \hline
n = 2 & - & -&- & 0 \\[.01in]
n = 3 & - & 0  & 0 & 1/2\\[.01in]
n = 4 & 0 & (2p+1)/(2p+2) & (p+2)/(2p+2)& 1 - (p/(4(p^2+p+1)) \\[.01in]
n \ge 5 & 1 & 1& 1 & 1
\end{array} \]
Now, using Lemma~\ref{alphaformulas}, we
compute~$\alpha_1^{(n)}$ and $\alpha_2^{(n)}$:
\[ \begin{array}{c|c|c}
 & \alpha_1^{(n)} & \alpha_2^{(n)} \\[.01in] \hline
n = 2 & 0 & 1/(2p+2) \\[.01in]
n = 3 & 1/(p+1) & (p+2)/(2p+2)\\[.01in]
n = 4 & 1 - (p^3/(2(p+1)(p^2+p+1))) & 1 - (p^3/(4(p+1)(p^3+p^2+p+1)))\\[.01in]
n \ge 5 & 1 & 1
\end{array} \]
Finally, we compute $\rho_n(p)$ using~(\ref{rhoformula}), yielding 
the values stated in Theorem~\ref{mainlocal}.

\vspace{.1in}
Note that our proof of Theorem~\ref{mainlocal} also yields a (recursive) algorithm to determine whether 
a quadratic form over $\Q_p$ is isotropic.  Tracing through the algorithm, we see that, for a quadratic form of nonzero discriminant,  only finitely many recursive iterations are possible 
(since we may organize the algorithm so that at each such iteration the discriminant 
valuation is reduced), i.e., the algorithm always terminates.  
In particular, when $n\geq 5$, our algorithm always yields a zero for any $n$-ary quadratic form of nonzero discriminant; hence every nondegenerate quadratic form in $n\geq 5$ variables is isotropic.

\section{The density of $n$-ary quadratic forms over $\R$ that are indefinite: Proof of Theorem~\ref{mainreal}}\label{sec:R}

\subsection{Preliminaries on the Gaussian Orthogonal Ensemble (\GOE)}

We wish to calculate the probability $\rho_{n}(\infty)$ that a real
symmetric matrix $M$ from the $n$-dimensional \GOE\ has an indefinite
spectrum.  The distribution of matrix entries in the \GOE\ is
invariant under orthogonal transformations. Since real symmetric
matrices can be diagonalised by an orthogonal transformation, the
\GOE\ measure can be written directly in terms of the eigenvalues
$\lambda(M)$, yielding the distribution
\begin{equation}
 \mathbb{P}\bigl(\lambda(M) \in [\lambda+ \D \lambda)\bigr)=\frac{1}{Z^{\mathrm{GOE}}_{n}} |\Delta(\lambda)|  \prod_{i=1}^n e^{-\frac{1}{4} \lambda_i^2}\D\lambda_i;
\end{equation} 
here 
$$\Delta(\lambda) := \prod_{1\leq i<j\leq n}(\lambda_j-\lambda_i) = \det(\varphi_i(\lambda_j)),
$$ where $(\varphi_i(\lambda_j))=(\lambda_j^{i-1}$) is a Vandermonde
matrix, and the normalizing factor ${Z^{\mathrm{GOE}}_{n}}$ is given
by
\begin{equation}
{Z^{\mathrm{GOE}}_{n}} =
n!(2\pi)^{\frac{n}{2}}2^{(n(n-1)/4 + n/2)}\prod_{j=1}^n \frac{\Gamma(\frac{j}{2})}{\Gamma(\frac{1}{2})}.
\end{equation}
See, for example, \cite[(2.5.4)]{AGZ}. 

Note that the probability that the matrix $M$ is indefinite is related
to the probability $p^+_{n}$ that all its eigenvalues are positive by
\begin{equation}
\rho_{n}(\infty) = 1 - \mathbb{P}(\mathrm{positive~definite})
                   - \mathbb{P}(\mathrm{negative~definite}) = 1-2p^+_{n},
\end{equation}
where the second equality follows by symmetry.  Below we will
calculate $p^+_{n}$, and hence obtain the value of $\rho_{n}(\infty)$.

\subsection{de Bruijn's identity}\label{sec:db}

We recall a useful result from \cite[\S4]{debruijn} for
calculating determinantal integrals of the type we will need. As a
generalisation of an expression for the volume of the space of
symmetric unitary matrices, de Bruijn considered integrals of the
form:
\begin{equation}
\Omega =\underset{a\leq x_1\leq \cdots \leq x_n \leq b} {\int\cdots\int}\underset{1\leq i,j\leq n}\det(\varphi_i(x_j)) \D x_1\dots\D x_n . \label{omega}
\end{equation}
Recall that the Pfaffian of a skew-symmetric matrix $A=(a_{ij})$ is
given by
\begin{align}{\rm Pf}(A) =
  \sum_{\tau}\operatorname{sgn}(\tau)a_{i_1,j_1}a_{i_2,j_2}\cdots
  a_{i_s,j_s}\label{pf},\end{align} where $\tau$ ranges over all
partitions
$$\tau =\{(i_1,j_1), (i_2,j_2),\dots
(i_s,j_s)\}$$ of $n=2s$ where $i_k<i_{k+1}$ and $i_k<j_k$. The sign is of the
corresponding permutation $$\tau =\begin{bmatrix} 1 & 2 & 3 & 4 &
\cdots & 2s\\ i_1 & j_1 & i_2 & j_2 & \cdots & j_{s} \end{bmatrix}.$$

The integral (\ref{omega}) may be rewritten as the Pfaffian of either
an $n \times n$ skew-symmetric matrix if $n$ is even, or an $(n+1)
\times (n+1)$ skew-symmetric matrix if $n$ is odd.  More precisely,
let $n':=2 \lceil n/2 \rceil$; then we have $\Omega = \mathrm{Pf}(A)$,
where $A$ is the $n'\times n'$ skew-symmetric matrix whose
$(i,j)$-entry $a_{ij}$ is given for $i < j$ by
\begin{equation}
a_{ij} = \begin{cases} \int_a^b\int_a^b
  \mathrm{sign}(y-x)\varphi_i(x)\varphi_j(y) \D x \D y & \mbox{if}~ i
  < j \leq n; \\[.075in] \int_a^b\varphi_j(x) \D x & \mbox{if}~
  i<j=n+1.\\[.075in]
\end{cases} \label{Adef}
\end{equation}
The second case occurs only when $n$ is odd. Note that this holds for
a general measure $\D x$; below, we will use $\D x =
e^{-\lambda^2/4}\D\lambda$, where $\D \lambda$ is the Lebesgue measure
on $\R$.

The Pfaffian form of the integral is found by expanding the
determinant and using a signature function to keep track of the signs
and the ordering of the $x_i$. This signature function of $n$
variables can be broken up into a sum of products of two-variable
pieces (and a one-variable piece if $n$ is odd) and thus the integral
can be factorised into a sum of products of two (and one) dimensional
integrals which is recognised as of the form (\ref{pf}) for a matrix
with entries (\ref{Adef}).

\subsection{Calculation of $\rho_n^{\infty}$}

For a matrix $M$ from the \GOE, the joint distribution of the
eigenvalues $\lambda_1(M)\leq \lambda_2(M) \leq \cdots \leq
\lambda_n(M) $ is given by
\begin{equation}
 \frac{n!}{Z^{\mathrm{GOE}}_{n}} \mathbf{1}_{\lambda_1\leq \lambda_2\leq \cdots\leq \lambda_n} |\Delta(\lambda)|  \prod_{i=1}^n e^{-\frac{1}{4} \lambda_i^2}\D\lambda_i.
\end{equation}
The ordering in the domain of integration below means that we can
replace $|\Delta(\lambda)|$ by $\Delta(\lambda)$.  It then follows
that $p_{n}^+$ is given by the integral
\begin{align}
 p_{n}^+ &=\frac{n!}{Z^{\mathrm{GOE}}_{n}} \underset{0\leq \lambda_1\leq \cdots \leq \lambda_n \leq \infty} {\int\cdots\int} \Delta(\lambda) \prod_{i=1}^ne^{-\frac{1}{4} \lambda_i^2}\D\lambda_i \nonumber\\
 &=  \frac{n!}{Z^{\mathrm{GOE}}_{n}}\underset{0\leq \lambda_1\leq \cdots \leq \lambda_n \leq \infty} {\int\cdots\int} \det(\varphi_i(\lambda_j)) \prod_{i=1}^ne^{-\frac{1}{4} \lambda_i^2}\D\lambda_i \nonumber \\
 &=\frac{n!}{Z^{\mathrm{GOE}}_{n}} \mathrm{Pf}(A), \label{prob}
 \end{align}
where the last equality follows from the result of
\S\ref{sec:db}. Here, $A=(a_{ij})$, where for $i<j\le n$ 
we define
\begin{align}\label{matrixel}
a_{ij} &= \int_0^{\infty}\int_0^\infty \mathrm{sign}(y-x)x^{i-1}y^{j-1} e^{-\frac{x^2+y^2}{4}}\D x \D y \nonumber\\
&= 2^{i+j-2}\Gamma\left(\frac{i+j}{2}\right)\left(\beta_{\frac{1}{2}}\left(\frac{i}{2},\frac{j}{2}\right)- \beta_{\frac{1}{2}}\left(\frac{j}{2},\frac{i}{2}\right)\right),
\end{align}
and for $n$ odd we also set $a_{i,n+1} = 2^{i-1}\Gamma(\frac{i}{2})$.
Here the gamma and incomplete beta functions are as defined in
\S\ref{sec:beta}.  From the resulting skew-symmetric matrix $A$, we
may evaluate (\ref{prob}) to determine $\rho_{n}(\infty)$, yielding
Theorem~\ref{mainreal}.  Explicit values of $\rho_n(\infty)$ are
displayed in Table~\ref{table:GOE-results} for $n\le8$.


\begin{remark} \label{degrmk}\rm
  It is easily shown that the matrix entries $a_{ij}$ in
  Theorem~\ref{mainreal} are of the form $x$ or $x \sqrt{\pi}$ for $x
  \in \Q(\sqrt{2})$, in accordance with whether $i+j$ is even or odd. Let $s = \lceil
  n/2 \rceil$, so that $A$ is a $2s \times 2s$ matrix.  Then after
  re-ordering the rows and columns we have
  \[ {\rm Pf}(A) = \pm {\rm Pf} \begin{pmatrix} A_1 & \sqrt{\pi} A_2
    \\ -\sqrt{\pi} A_2^t & A_3 \end{pmatrix} = \pm \pi^{s/2} {\rm
    Pf} \begin{pmatrix} A_1 & A_2 \\ -A_2^t & \pi^{-1}
    A_3 \end{pmatrix} \] where $A_1, A_2$ and $A_3$ are $s \times s$
  matrices with entries in $\Q(\sqrt{2})$. Since $\prod_{m=1}^n
  \Gamma(m/2) = \pi^{s/2} y$ for some $y \in \Q$, it follows by
  Theorem~\ref{mainreal} and the definition of the Pfaffian that
  $\rho_n(\infty)$ is a polynomial in $\pi^{-1}$ having coefficients in
  $\Q(\sqrt{2})$ and degree at most $\lfloor s/2 \rfloor = \lfloor
  (n+1)/4 \rfloor$.
\end{remark}

\subsection*{Acknowledgments}
We thank Carlos Beltran, Jonathan Hanke, Peter Sarnak, and Terence Tao for helpful
conversations.  The first author (Bhargava) was supported by a Simons
Investigator Grant and NSF grant~DMS-1001828; the second (Cremona) and
fifth (Keating) were supported by EPSRC Programme Grant EP/K034383/1
LMF: L-Functions and Modular Forms; the fifth (Keating) was also supported by a grant from The Leverhulme Trust, a Royal Society Wolfson Merit Award, a Royal Society Leverhulme Senior Research Fellowship, and by the Air Force Office of Scientific Research, Air Force Material Command, USAF, under grant number FA8655-10-1-3088.

\end{document}